\documentclass[11pt,leqno]{article}
\usepackage{graphicx, amsfonts, amsthm, amsxtra, amssymb, verbatim}
\usepackage[mathscr]{euscript}

\begin{document}
\def\Sp{{\rm Sp}}
\def\Ra{{\rm Ra}}
\def\GM{{\rm GM}}

\def\per{{\rm X}}      
\def\perr{{\sf q}}        
\def\perdo{{\cal K}}   
\def\sfl{{\mathrm F}} 
\def\sp{{\mathbb S}}  

\newcommand\diff[1]{\frac{d #1}{dz}} 
\def\End{{\rm End}}              

\def\sing{{\rm Sing}}            
\def\spec{{\rm Spec}}            
\def\cha{{\rm char}}             
\def\Gal{{\rm Gal}}              
\def\jacob{{\rm jacob}}          
\def\tjurina{{\rm tjurina}}      
\newcommand\Pn[1]{\mathbb{P}^{#1}}   
\def\Ff{\mathbb{F}}                  
\def\Z{\mathbb{Z}}                   
\def\Gm{\mathbb{G}_m}                 
\def\Q{\mathbb{Q}}                   
\def\C{\mathbb{C}}                   
\def\O{{\cal O}}                     
\def\as{\mathbb{U}}                  
\def\ring{{\mathsf R}}                         
\def\R{\mathbb{R}}                   
\def\N{\mathbb{N}}                   
\def\A{\mathbb{A}}                   
\def\uhp{{\mathbb H}}                
\newcommand\ep[1]{e^{\frac{2\pi i}{#1}}}
\newcommand\HH[2]{H^{#2}(#1)}        
\def\Mat{{\rm Mat}}              
\newcommand{\mat}[4]{
     \begin{pmatrix}
            #1 & #2 \\
            #3 & #4
       \end{pmatrix}
    }                                
\newcommand{\matt}[2]{
     \begin{pmatrix}                 
            #1   \\
            #2
       \end{pmatrix}
    }
\def\ker{{\rm ker}}              
\def\cl{{\rm cl}}                
\def\dR{{\rm dR}}                

\def\hc{{\mathsf H}}                 
\def\Hb{{\cal H}}                    
\def\GL{{\rm GL}}                
\def\pese{{\sf P}}                  
\def\pedo{{\cal  P}}                  
\def\PP{\tilde{\cal P}}              
\def\cm {{\cal C}}                   
\def\K{{\mathbb K}}                  
\def\k{{\mathsf k}}                  
\def\F{{\cal F}}                     
\def\M{{\cal M}}
\def\RR{{\cal R}}
\newcommand\Hi[1]{\mathbb{P}^{#1}_\infty}
\def\pt{\mathbb{C}[t]}               
\def\W{{\cal W}}                     
\def\gr{{\rm Gr}}                
\def\Im{{\rm Im}}                
\def\Re{{\rm Re}}                
\def\depth{{\rm depth}}
\newcommand\SL[2]{{\rm SL}(#1, #2)}    
\newcommand\PSL[2]{{\rm PSL}(#1, #2)}  
\def\Resi{{\rm Resi}}              

\def\L{{\cal L}}                     
\def\Aut{{\rm Aut}}              
\def\any{R}                          
\newcommand\ovl[1]{\overline{#1}}    

\def\T{{\cal T }}                    
\def\tr{{\mathsf t}}                 
\newcommand\mf[2]{{M}^{#1}_{#2}}     
\newcommand\mfn[2]{{\tilde M}^{#1}_{#2}}     
\newcommand\bn[2]{\binom{#1}{#2}}    
\def\ja{{\rm j}}                 
\def\Sc{\mathsf{S}}                  
\newcommand\es[1]{g_{#1}}            
\newcommand\V{{\mathsf V}}           
\newcommand\WW{{\mathsf W}}          
\newcommand\Ss{{\cal O}}             
\def\rank{{\rm rank}}                
\def\Dif{{\cal D}}                   
\def\gcd{{\rm gcd}}                  
\def\zedi{{\rm ZD}}                  
\def\BM{{\mathsf H}}                 
\def\plf{{\sf pl}}                             
\def\sgn{{\rm sgn}}                      
\def\diag{{\rm diag}}                   
\def\hodge{{\rm Hodge}}
\def\HF{{\sf F}}                                
\def\WF{{\sf W}}                               
\def\HV{{\sf HV}}                                
\def\pol{{\rm pole}}                               
\def\bafi{{\sf r}}
\def\codim{{\rm codim}}                               
\def\id{{\rm id}}                               
\def\gms{{\sf M}}                           
\def\Iso{{\rm Iso}}                           

\def\hl{{\rm L}}    
\def\bcov{{\rm BCOV}}
\def\imF{{\rm F}}
\def\imG{{\rm G}}

\def\ord{{\rm ord}}
\newcommand\FF{F(a_1,\cdots,a_n|z)} 
\newcommand*\pFqskip{8mu}
\catcode`,\active
\newcommand*\pFq{\begingroup
        \catcode`\,\active
        \def ,{\mskip\pFqskip\relax}%
        \dopFq
}
\catcode`\,12
\def\dopFq#1#2#3#4#5{%
        {}_{#1}F_{#2}\biggl(\genfrac..{0pt}{}{#3}{#4};#5\biggr)%
        \endgroup
}

\def\losu {{G}}
\newcommand\dwork[1]{\delta_{#1}}  
\newtheorem{theo}{Theorem}
\newtheorem{exam}{Example}
\newtheorem{coro}{Corollary}
\newtheorem{defi}{Definition}
\newtheorem{prob}{Problem}
\newtheorem{lemm}{Lemma}
\newtheorem{prop}{Proposition}
\newtheorem{rem}{Remark}
\newtheorem{conj}{Conjecture}
\newtheorem{calc}{}

\begin{center}
{\LARGE\bf Modular-type functions attached to Calabi-Yau varieties: integrality properties
}
\\
\vspace{.25in} {\large {\sc Hossein Movasati and Khosro M. Shokri }}
\\
Instituto de Matem\'atica Pura e Aplicada, IMPA,\\
Estrada Dona Castorina, 110,\\
22460-320, Rio de Janeiro, RJ, Brazil,\\
{\tt hossein@impa.br,\ \ khosro@impa.br} \\
\end{center}

\begin{abstract}
We study the integrality properties of the coefficients of the mirror map attached to
 the generalized hypergeometric function  $_{n}F_{n-1}$ with rational parameters and with a maximal unipotent monodromy. 
We present a conjecture on the $p$-integrality of the mirror map which can be verified experimentally. 
We prove its consequence on the $N$-integrality of the mirror map  for the particular cases $1\leq n\leq 4$. 
 For $n=2$ we obtain the  Takeuchi's classification of arithmetic triangle 
groups with a cusp, and for $n=4$ we prove that  $14$ examples of hypergeometric Calabi-Yau equations are the full classification
of hypergeometric mirror maps with integral coefficients. 
As a by-product we get the integrality of the corresponding algebra  of modular-type functions. These
are natural generalizations of the algebra of classical modular and quasi-modular forms in the case $n=2$. 
\end{abstract}
\section{Introduction}
The integrality of the coefficients of the mirror map is a central problem in the arithmetic of Calabi-Yau varieties and it has been 
investigated in many recent articles  \cite{liya, zu02, stsa-09, ksv, kr10}. The central tool in all these works has been the so 
called Dwork method developed in \cite{dw69, dwo73}. It seems to us that the full consequences of Dwork method has not been explored, that is, 
to classify all hypergeometric differential equations with a maximal unipotent monodromy  whose mirror maps have integral coefficients.
In this article, we fill this gap and give a  computable condition on the parameters of a hypergeometric function which conjecturally computes all
the primes which appear in the denominators of the coefficients of the mirror map. We verify this conjecture and some of its consequences 
in many particular cases and give many computational evidence for its validness.

Let $a_i,\ i=1,2,\ldots,n$ be rational numbers, $0<a_i<1$,   $a=(a_1,a_2,\ldots,a_n)$ and
$$
F(a| z):=\   _{n}F_{n-1}(a_1,\ldots,a_n; 1,1,\ldots,1 | z)= \sum_{k=0}^\infty \frac{(a_1)_k\dots(a_n)_k}{k!^n} \, z^k, \quad |z|<1
$$
 be the holomorphic solution of the generalized hypergeometric differential equation 
\begin{equation}
\label{14fev}
\theta^n-z(\theta+a_1)(\theta+a_2)\cdots(\theta+a_n)=0
\end{equation}
where $(a_i)_k=a_i(a_i+1)(a_i+2)...(a_i+k-1),\,(a_i)_0 = 1,$ is the Pochhammer symbol and $\theta=z\frac{d}{dz}$. 
The first logarithmic solution in the Frobenius basis around $z=0$ has the form $\losu(a|z)+F(a|z)\log z$,  
where
\begin{equation}
\label{logsol}
\losu(a| z)=\sum_{k=1}^\infty \frac{(a_1)_k\cdots(a_n)_k}{(k!)^n}\big{[}\sum_{j=1}^ n\sum_{i=0}^{k-1}(\frac{1}{a_j+i}-\frac{1}{1+i})\big{]}z^k.
\end{equation}
The mirror map 
$$
q(a| z)=:z\exp(\dfrac{\losu(a|z)}{F(a| z)} ),
$$ is a natural generalization of the
Schwarz function.

For a prime $p$ and a formal power series $f\in\Q[[z]]$ with rational coefficients we say that $f$ is $p$-integral if $p$ does not appear 
in the denominator of its coefficients or equivalently
$f$ induces a formal power series in $\Z_p[[z]]$.  For a rational number $x$ such that $p$
does not divide the denominator of $x$, we define 
$$
\dwork{p}(x):=\frac{x+x_0}{p},
$$ 
where $0\leq x_0\leq p-1$ is the unique integer such that $p$ does not
divide the denominator of $\dwork{p}(x)$. We call $\dwork{p}$ the Dwork operator.
Throughout the text we assume all parameters $a_i$'s in the hypergeometric equation \eqref{14fev} are $p$- integers and such $p$
is called a good prime.

\begin{conj}
\label{main}
Let $q(a| z)$ be the mirror map of the generalized hypergeometric function, defined as above. Let 
Then $q(a| z)$ is $p$-integral if and only if
\begin{equation}
 \label{5mar}
\{\dwork{p}(a_1),\dwork{p}(a_2)\}=\{a_1,a_2\}, \hbox{ or }\{1-a_1,1-a_2\} \hbox{ for }  \quad n=2
\end{equation}
and
\begin{equation}
 \label{5mar2013}
\{\dwork{p}(a_1),\dwork{p}(a_2), \dwork{p}(a_3),\ldots, \dwork{p}(a_n)\}=\{a_1,a_2,a_3,\ldots, a_n\},\ \ \hbox{ for }   n\not= 2.
\end{equation}
\end{conj}
The conjecture for $n=1$ is an easy exercise.  
Due to the Euler identity for the Gauss hypergeometric function,  the case $n=2$ appears different from other cases, see
\S \ref{n=2}.   
For general $n$ we prove the only if part  and give computational 
evidence for the validity of the other direction, see \S\ref{23may2013}.  
In the literature one is mainly  interested to classify all $N$- integral mirror maps. This means that there is a natural number $N$ such that 
$q(a|Nz)$ has integral coefficients.\\

\textbf{Remark}.
One can see that the order of bad primes in  the coefficients of $F$, $G$
and consequently in $q$ has at most polynomial growth. Hence there is a number $N_p$ such that, $q(a|N_pz)$ is $p$- integral.
Therefore $p$ integrality of $q(a|z)$ for almost all $p$ implies that $q(a|z)$ is $N$- integral.

The Conjecture \ref{main} easily
implies
\begin{conj}
\label{main2}
 The mirror map $q(a| z)$ is $N$-integral if and only if for any good prime $p$ we have (\ref{5mar}) for $n=2$ and (\ref{5mar2013}) for 
$n\geq 3$.
\end{conj}

\begin{theo}
\label{maintheo}
We have
\begin{enumerate} 
\item
For an arbitrary $n$ the only if part of Conjecture \ref{main} and Conjecture \ref{main2} are true. 
 \item 
Conjecture \ref{main2} for $n=1,2,3,4$ is true.
\end{enumerate}
\end{theo}
The first  part of the above theorem together with an exact formula for the smallest number $N$ such that $q(a|Nz)$  is integral, was a 
conjecture in the context of mirror symmetry  The case $a_i=\frac{i}{n+1},\ i=1,2,\ldots,n$, where $n+1$ is a prime number (resp. 
power of a prime number) is proved by Lian-Yau in \cite{liya} (resp.by  Zudilin in \cite{zu02}).  
A general format is formulated by Zudilin in \cite{zu02} and  is  proved by  Krattenthaler-Rivoal in \cite{kr10}. 
According to Conjecture \ref{main2} the $N$-integrality of
the mirror map implies that  for a fixed good prime $p$ the map $\dwork{p}$ acts on the  set $\{a_1,\cdots,a_n\}$ as a permutation. This set is decomposed into subsets of 
cardinality  $\phi(m_i),\ m_i>1, \ i=1,2,\ldots, k$, where 
$\phi$ is the arithmetic Euler function.
This is done  according to whether two elements have the same denominator or not. The numbers $m_i,\ i=1,2,\ldots$ are all possible denominators in the  set $\{a_1,\cdots,a_n\}$.
Such a decomposition is invariant under the  permutation  induced by $\dwork{p}$.  We conclude that the number of $N$-integral mirror maps is equal to the number of decompositions:
\begin{equation}
 \label{14feb2}
n=\phi(m_1)+\cdots+\phi(m_k)
\end{equation}
Like in the classical case of
the partition of a number with natural numbers we have the following generating function:
$$
24x^2-1+\prod_{m=2}^\infty \frac{1}{1-x^{\phi(m)}}=x+28x^2+4x^3+14x^4+ 14x^5+40x^6+40x^7+\cdots.
$$
(for the coefficient of $x^2$ see Table 1). 
Conjecture \ref{main2} says that the coefficient of $x^n$  counts the number 
of $N$-integral mirror maps with $n$-parameters $a_1,a_2,\ldots,a_n$.
An immediate consequence of (\ref{14feb2}) is that the number of $N$- integral mirror maps for $n=2\ell+1,\ \ell\geq 2$ exactly equals with $n=2\ell$ ones. 
Since $\phi(m)$, for $m>2$ is even, for $n=2\ell+1$ in the right hand side of \eqref{14feb2}, one of the $m_i$'s  is $2$ or equivalently one of $a_i$ is $\frac{1}{2}$. 
After canceling this from both sides we return to the 
case $n=2\ell$. 
Therefore, the number of $N$-integral mirror maps for $n=2\ell$ and $2\ell+1$ are in a one to one correspondence.

For any fixed $n$, we can use conjecture  \ref{main2} and  classify all $N$-integral mirror maps. 
For instance for $n=2$, the  mirror map is $N$-integral if and only if  $\{a_1,a_2\}$ belongs to the list in Table 1.  
The first four cases correspond to $\{\dwork{p}(a_1),\dwork{p}(a_2)\}=\{a_1,a_2\}$ and the others correspond to 
$\{\dwork{p}(a_1),\dwork{p}(a_2)\}=\{1-a_1,1-a_2\}$. If we set $\{a_1,a_2\}=\{\frac{1}{2}(1\pm \frac{1}{m_1}-\frac{1}{m_2})\}$ with
$m_1,m_2\in\N$ and $\frac{1}{m_1}+\frac{1}{m_2}>1$ then the monodromy group of the Gauss hypergeometric equation (\ref{14fev}) is a triangle group
of type $(m_1,m_2,\infty)$. In this case the above classification is reduced to the Takeuchi's classification in \cite{tak77} of arithmetic triangle 
groups with a cusp, see \cite{hokh2}. 

For $n\geq 4$ even  and with an $N$-integral mirror map $q(a| z)$, the set $\{a_1,a_2,\ldots,a_n\}$
is conjecturally invariant under $x\mapsto 1-x$, see \S\ref{dworkmap} and so we may identify $a=(a_1,a_2,\ldots,a_n)$ 
with its $\frac{n}{2}$ elements in the interval 
$[0,\frac{1}{2}]$. Below we just list these elements.  The case $n=4$ has many applications in mirror symmetry. 
We find that $q(a|z)$ is $N$-integral if and only if $(a_1,a_2,a_3,a_4)$ belongs to the well-known $14$ 
hypergeometric cases of Calabi-Yau equations. The first two elements $(a_1,a_2)$ are given in Table 1.
Note that in \cite{dormor}, these $14$ cases are classified through properties of the monodromy group of the differential
equation (\ref{14fev}) 
which comes from the variation of Hodge structures of Calabi-Yau varieties. The fact that the corresponding mirror maps are 
$N$-integral and are the only ones with this property is a non-trivial statement.
For $n=6$ we find  $40$ examples of $N$-integral mirror maps. In this case $(a_1,a_2,a_3)$ are given in Table 1. Finding $40$ examples of 
one parameter families of Calabi-Yau varieties of dimension $5$ may be done in a similar way as in the $n=4$ case. This is left for a future 
work. Note that for an arbitrary $n$ the Picard-Fuchs equation of the Dwork family $x_1^{n+1}+\cdots+x_{n+1}^{n+1}-(n+1)z^{-\frac{1}{n+1}} x_1x_2\ldots x_{n+1}=0$
corresponds to the case $a_i=\frac{i}{n+1},\ \ i=1,2,\ldots,n$.

The discussion for $n=2$ suggests that the $N$-integrality of the mirror map is related  to the arithmeticity of 
the monodromy group of (\ref{14fev}). However, note that for $n=4$ it is proved that the seven of the fourteen examples are thin, that is,
they have infinite index in $\Sp (4,\Z)$, see \cite{brth-12}. Note also that, we know a complete classification of the Zariski closure
of the monodromy group of (\ref{14fev}), see \cite{beu89}.
Since in \cite{ho22} we have constructed a new kind of modular forms theory
for these examples of thin groups, see also \S \ref{mtp}, it is reasonable to weaken the notion of arithmeticity so that the monodromy group
of (\ref{14fev}) with an $N$-integral mirror map becomes arithmetic in this new sense. For a discussion of thin groups
and their applications see Sarnak's lecture notes \cite{sar12} and the references therein.  
Finding the  smallest $N$ is not a trivial problem and is discussed in the articles \cite{liya, zu02,  kr10} for particular cases including the 14
cases in Table 1. For the following corollary we assume this $N$.

\begin{coro}
\label{12june2013}
Let $n=4$ and consider one of the 14 cases in Table 1. Let also $z(q)$ be the inverse of the mirror map $q(z):=\frac{1}{N}q(a|Nz)$ and 
\begin{align*}
&u_0(z):=z,\quad u_i(z):=\theta^{i}(F(Nz)),\ i=0,1,2,3, \\ 
&u_{i+4}(z):=F(Nz)\theta^i (\losu(Nz))-\losu(Nz)\theta ^i(F(Nz)),\quad i=1,2.
\end{align*} 
We have 
\begin{equation}
\label{zud}
u_0(q),u_i(z(q))\in \Z[[q]],\ \ i=0,1,\ldots,6.
\end{equation}
\end{coro}
The field of modular-type functions is generated by (\ref{zud}) and it is invariant under $q\frac{\partial}{\partial q}$. 
There are many analogies between $\Q(u_i(z(q)),\ i=0,1,\ldots,6)$ and  the field of quasi-modular forms, see for instance \cite{ho22}. 
 This includes functional equations with
respect to the monodromy group of (\ref{14fev}), the corresponding Halphen equation and so on.
However, note that the former field is of transcendental degree 
$3$, whereas this new field is of transcendental degree $7$.  For $n=2$ a similar discussion as in Corollary \ref{12june2013} leads us to the 
theory of (quasi) modular forms for the monodromy group of (\ref{14fev}), 
see for instance \cite{hokh2}.

We would like to thank M. Belolipetsky for taking our attention to the works on thin groups. Thanks also go to P. Sarnak, S. T. Yau, F. Beukers and 
W. Zudilin for comments on the first draft of the present text. 
The second author would like to thank CNPq-Brazil for financial support and IMPA for its lovely research ambient.  In the the final steps of the present text, 
F. Beukers informed us about the article \cite{roq} in which the author 
proves Conjecture \ref{10d2012} in \S \ref{1june2013} using results in differential Galois theory. As a consequence the if part of the Conjecture 
\ref{main2} follows easily. 
Our proof of the Conjecture \ref{main2} for $n=1,2,3,4$ is elementary and self-content.

{\tiny
\begin{center}
 \label{14cases}
\begin{tabular}{|c|}
\hline
$n=2$ \\ \hline
$
( 1/2 , 1/2 ),
( 2/3 , 1/3 ),
( 3/4 , 1/4 ),
( 5/6 , 1/6 ),
$
\\
$
( 1/6 , 1/6 ),
( 1/3 , 1/6 ),
( 1/2 , 1/6 ),
( 1/3 , 1/3 ),
( 2/3 , 2/3 ),
( 1/4 , 1/4 ),
( 1/2 , 1/4 ),
( 3/4 , 1/2 ),
( 3/4 , 3/4 ),
$
\\
$
( 1/2 , 1/3 ),
( 2/3 , 1/6 ),
( 2/3 , 1/2 ),
( 5/6 , 1/3 ),
( 5/6 , 1/2 ),
( 5/6 , 2/3 ),
( 5/6 , 5/6 ),
( 3/8 , 1/8 ),
( 5/8 , 1/8 ),
$
\\
$
( 7/8 , 3/8 ),
( 7/8 , 5/8 ),
( 5/12 , 1/12 ),
( 7/12 , 1/12 ),
( 11/12 , 5/12 ),
( 11/12 , 7/12 )
$\\ \hline 
$n=4$ \\ \hline
$
(1/2, 1/2), (1/3, 2/3), (1/4,1/2), ( 1/4, 1/4), (2/5, 1/5), (3/8, 1/8), (3/10, 1/10),$
\\
$
(1/2, 1/6), (1/2, 1/3), (1/3, 1/6),(1/6, 1/6),  (1/3, 1/4), (1/4, 1/6), (5/12, 1/12)
$\\ \hline
$n=6$ \\ \hline
$
(1/2, 1/2, 1/2),
(1/3, 1/3, 1/3),
(1/2, 1/2, 1/4),
( 1/2, 1/4, 1/4),
( 1/4, 1/4, 1/4),
( 1/2, 1/2, 1/3),
( 1/2, 1/3, 1/3),
$
\\
$
( 1/2, 1/2, 1/6),
( 1/2 ,1/3, 1/6),
( 1/3, 1/3, 1/6),
( 1/2, 1/6, 1/6),
( 1/3, 1/6, 1/6),
( 1/6, 1/6, 1/6),
( 3/7, 2/7, 1/7),
$
\\
$
( 1/2, 3/8, 1/8),
( 3/8 ,1/4 ,1/8),
( 4/9, 2/9 ,1/9),
( 1/2 ,2/5, 1/5),
( 1/2, 3/10 ,1/10)
( 1/2, 1/3 ,1/4),
(1/3 ,1/3, 1/4),
$
\\
$
( 1/3, 1/4 ,1/4),
( 1/2 ,1/4 ,1/6),
( 1/3, 1/4 ,1/6),
(1/4 ,1/4, 1/6),
(1/4 ,1/6 ,1/6),
( 1/2, 5/12, 1/12),
( 5/12, 1/3, 1/12),
$
\\
$
( 5/12, 1/4, 1/12),
( 5/12 ,1/6 ,1/12),
( 5/14, 3/14, 1/14),
(2/5, 1/3, 1/5),
(7/18 ,5/18, 1/18),
( 2/5, 1/4, 1/5),
(3/10, 1/4, 1/10),
$
\\
$
( 3/8, 1/3, 1/8),
( 3/8, 1/6 ,1/8),
( 2/5, 1/5, 1/6),
( 1/3 ,3/10 ,1/10),
( 3/10 ,1/6 ,1/10)
$\\ \hline
 \end{tabular}\\
\label{physics-mine}
Table 1: $N$-integral hypergeometric  mirror maps. 
\end{center}
}

\section{Dwork's theorem}
As far as we know, the main argument in the literature for proving the $p$- or $N$-integrality of mirror maps is the work of 
B. Dwork in  \cite{dwo73, dw69}.

\subsection{Dwork map}
\label{dworkmap}
\def\CC{\tilde \Z}
The Dwork map is defined in the following way
$$
\dwork{p}: \Z_p\to \Z_p,\  \ \sum_{s=0}^\infty x_sp^s \longmapsto 1+\sum_{s=0}^\infty x_{s+1}p^s,\ \ \ 0\leq x_s\leq p-1
$$
(and so $p\dwork{p}(x)-x=p-x_0$ for $x\in\Z_p$). 
Let $\CC_p$ be the set of $p$-integral rational numbers.
We have a natural embedding $
\CC_p\hookrightarrow \Z_p$. 
The map $\dwork{p}$ leaves $\CC_p$ invariant because for $x\in \CC_p$, $\dwork{p}(x)$ is the unique number such that
$p\dwork{p}(x)-x \in \Z\cap [0,p-1]$. 
For a rational number $x=\frac{x_1}{x_2},\ \ x_1,x_2\in\Z$ and a prime $p$ which does not divide $x_2$, we have
\begin{equation}
\label{121212}
\dwork{p}(x):=\frac{p^{-1}x_1 \hbox{ mod } x_2}{x_2},
\end{equation}
where $p^{-1}$ is the inverse of $p$ mod $x_2$ and $x_1$ and $x_2$ may have common factors. 
The denominators of $x$ and $\dwork{p}(x)$ are the same and 
$\dwork{p}(1-x)=1-\dwork{p}(x)$.  
 For any finite set of rational numbers, there is a finite decomposition of prime numbers such that in each class the 
function $\dwork{p}$ is independent of
the prime $p$.

The following proposition easily follows from  (\ref{121212}). It will play an important role
in the proof of Theorem \ref{maintheo}. 
\begin{prop}
 \label{27may2013}
Let $0< x=\frac{t}{sq^y}<1$ be a rational number,  where $q$ is a prime and $(q,s)=1$ and  
$t$ and $sq^y$ may have common factors. 
We have
\begin{enumerate}
\item
\label{24ap--}
For primes $p$ such that $p^{-1}\equiv -1\pmod {sq^{y}}$ we have $\dwork{p}(x)=1-x$. 
\item
\label{24ap-0}
If $y=0$, that is the denominator of $x$ is not divisible by $q$, then for   primes $p$ such $p^{-1}\equiv q \pmod s$ we have
$$
\dwork{p}(x)=qx-i,\ \hbox{ for some } i=0,1,2,\ldots,q-1.
$$
\item
\label{24ap-1}
For primes $p^{-1}\equiv s+q \pmod{ sq^y}$ we have
$$
\dwork{p}(x)=qx+\frac{r}{q^y}-i,\ \hbox{ for some } \ r=0,1,\ldots, q^y-1,\ i=0,1,\ldots,q
$$
\item
\label{24ap-2}
If $y\geq 1$, that is, the denominator of $x$ may be divisible by $q$, then  for any $0\leq m\leq y$ and prime $p$ with
$p^{-1}\equiv 1+q^{y-m}s \pmod{sq^y}$ either we have $\dwork{p}(x)=x$ or we have
$$
\dwork{p}(x)=x+\frac{r}{q^m}-i,\ \hbox{ for some }\ r=1,\ldots,q^m-1,\ i=0,1.
$$
\end{enumerate}
\end{prop}


\subsection{Dwork lemma and theorem on hypergeometric functions}
In this section we mention some lemmas and a theorem of Dwork. 
The following lemma is crucial in the argument of $N$-integrality of the mirror map. 
Since the non trivial application of this lemma was first given by
Dwork, it is known as  Dwork lemma, however,  Dwork himself in \cite{dw94} associates  it to Dieudonn\'{e}.
\begin{lemm}
\label{dieudonne}
Let $f(z)\in 1+z\Q_p[[z]]$. Then $f(z)\in 1+z\Z_p[[z]]$, if and only if 
$$
\frac{f(z^p)}{(f(z))^p}\in 1+p\Z_p[[z]].
$$
\end{lemm}
For a more general statement and the proof  see \cite{dw94}, p.54. For the following theorem the reader is referred to \cite{dw69, dwo73}.
\begin{theo}
\label{9/11}
Let $F,\losu$ as before. We have
$$
\frac{\losu(\dwork{p}(a)|z^p)}{F(\dwork{p}(a)|z^p)}\equiv p\frac{\losu(a|z)}{F(a|z)} \pmod {p\Z_p[[z]]}.
$$
\end{theo}

\subsection{Consequences of Dwork's theorem}
The  mirror map is the invertible function $q(a| z)=z\exp(\frac{\losu(a| z)}{F(a | z)})$. The inverse of the mirror map is important for 
the construction of modular-type functions. 
In this section we give conditions for integrality of the mirror map.

\begin{lemm}
\label{9/11cor}
 Let $a_1,a_2,\ldots,a_n\in\Q$ with the conditions \eqref{5mar} and \eqref{5mar2013} for a prime $p$. Then
$$
q(a|z)\in z\Z_p[[z]],
$$
\end{lemm}
\begin{proof}
 Let $f(z)= \dfrac{\losu(a|z)}{F( a|z)}$. Since $f(z)\in z\Q[[z]]$, so $\exp(f(z))\in 1+z\Q[[z]]$. 
Our assumption and Theorem \ref{9/11} imply  that
$$
f(z^p)-pf(z)=p\cdot g(z) ,\ \  \ \ g(z)\in z\Z_p[[z]].
$$
Since 
$\exp(p\,.g(z))=1+\sum_ {k=1}^\infty \frac{p^k}{k!}g(z)^k$ and $ord_p(k!)< k$, we find that 
$$
\frac{\exp f(z^p)}{(\exp f(z))^p}=\exp (p\,.g(z))\in 1+p\,z\Z_p[[z]].
$$
Now by applying Lemma \ref{dieudonne}, the result follows. 
\end{proof}

\begin{lemm}
\label{nopint}
If $q(a| z)$ is $p$-integral then
\begin{equation}
\label{4dec2012}
\frac{\losu(\dwork{p}(a) |z)}{F(\dwork{p}(a) |z)}\equiv \frac{\losu(a|z)}{F(a|z)} \pmod {p\Z_p[[z]]}.
\end{equation}
Furthermore if $q(a|z)$ is $p$-integral for all except a finite number of  primes then the above congruence is an equality for all good 
primes $p$.
\end{lemm}
We conjecture that in Lemma \ref{nopint}, the congruence (\ref{4dec2012}) is an equality. This together 
with  Conjecture \ref{10d2012} in \S \ref{1june2013}  imply Conjecture \ref{main}.  
\begin{proof}
Let 
$$
f(z):= \frac{\losu(a|z)}{F( a|z) },  \  \
f'(z):=\frac{\losu(\dwork{p}(a)|z)}{F(\dwork{p}(a)|z)}.
$$
 By Lemma \ref{dieudonne} $\exp(f)$ is $p$-integral if and only if $\exp(f(z^p)-pf(z))\in 1+p\Z_p[[z]]$ 
and by Theorem \ref{9/11} $f'(z^p)\equiv  pf(z) \pmod {p \Z_p[[z]]}$.  Combining these two facts
$$
\exp(f(z^p)-f'(z^p)) \in 1+p\,z\Z_p[[z]].
$$
or $f(z^p)-f'(z^p)=\log (1+p\,zg(z))$, for some $g(z)\in \Z_p[[z]]$. But 
$$
\log(1+p\, z\, g(z))=\sum _{n=1}^\infty (-1)^n\frac{p^nz^ng(z)^n}{n}\in p\,z\Z_p[[z]]
$$
Now, let us prove the second part. Since the number of bad primes is finite, our hypothesis implies that 
$q(a| z)$ is $p$-integral for all good primes except a finite number which may includes bad primes. 
Let $c$ is the common factor of the denominators of $a_i$'s and let $r$ be a prime residue of $c$. 
From \eqref{121212} for $p\equiv r \pmod c$ the the value of $\dwork{p}(a_i)$ is independent of the special member of this class of primes. 
On the other hand by assumption congruency \eqref{4dec2012} holds for almost all primes of this class and so  it must be equality. 
Now running over all prime residues $r$ gives the result. 
\end{proof}

\begin{rem}\rm
If the congruence (\ref{4dec2012}) does not happen then by Lemma \ref{nopint}, $q(a|z)$ is not $p$-integral.
In fact this turns out to be a fast way to check the non $p$-integrality than checking the non $p$-integrality of $q(a|z)$ directly. 
For instance, for $n=2$, $p=101,  a_1=169/330,\ \ a_2=139/330$ the truncated $q(a| z) \mod z^{k},\ k<101$ is $p$-integral but $q(a| z)$ is 
not $p$-integral because the congruency (\ref{4dec2012}) fails at the power $z^2$. 
\end{rem}

\begin{lemm}
\label{lem27may}
 If the mirror map $q(a| z)$ is $p$-integral then $q(\dwork{p}(a)| z)$ is also $p$-integral.
\end{lemm}
\begin{proof}
 The proof follows directly from the congruency (\ref{4dec2012}). Note that, if $f\in p\Z_p[[z]]$, then $e^{f}\in\Z_p[[z]]$.  
\end{proof}
\subsection{Proof of Theorem \ref{maintheo}, part 1 }
\label{proofs}
The only if part of both conjectures \ref{main} and \ref{main2} 
follows from Lemma \ref{9/11cor}. So, it remains to prove Conjecture \ref{main2} for $n=1,2,3,4$
\section{A problem in computational commutative algebra}
\label{1june2013}
The only missing step for the proof of Conjecture \ref{main2} is a solution to the following conjecture:
\begin{conj}
\label{10d2012}
Let $n\not =2$ and $a_i,b_i\in\Q,\ \ i=1,2,\ldots,n$ with $0<a_i,b_i<1$, we have the equality of formal power series
\begin{equation}
\label{9dec2012}
\frac{\losu(b_1,b_2,\ldots,b_n|z)}{F(b_1,b_2,\ldots,b_n |z)}=\frac{\losu( a_1,a_2,\ldots,a_n |z)}{F(  a_1,a_2,\ldots,a_n |z)}, 
\end{equation}
if and only if 
\begin{equation}
\label{9d2012}
\{b_1,b_2,\ldots, b_n\}=\{a_1,a_2,\ldots,a_n\}. 
\end{equation} 
\end{conj}
Let 
$$
\frac{\losu(a| z)}{F(a| z)}=\sum_{i=1}^\infty C_k(a)z^k,\ \ \ C_k(a)\in\Q[a].
$$
and 
$I_{n,m}$ be the ideal in $\Q[a,b]$ generated by $C_{k}(a)-C_k(b),\ \ k=1,2,\ldots,m$ and $I_n=I_{n,\infty}$.
The above conjecture over $\C$ is false, that is, the variety given by $I_n$ has many components other than those obtained by the 
permutation (\ref{9d2012}). Therefore, the above conjecture is equivalent to say that such extra components have no non-trivial  $\Q$-rational points.
We are planing to write a separate article on this topic.
In this section we will try to avoid Conjecture \ref{10d2012} in such a general context. 
Instead, we will use the structure of the operator $\dwork{p}$ in order 
to reduce the number of variables  and try to solve Conjecture \ref{10d2012} for such particular cases.
Below, we are going to write down the primary decomposition of $I_{n,m}$ in the ring $\Q[ a, b]$ for many particular cases. 
We have used the  Gianni-Trager-Zacharias algorithm implemented in {\tt Singular} under the command {\tt primdecGTZ}, see \cite{GPS01}.
The corresponding computer codes can be found in the first author's web page.  
\subsection{The case $n=2$}
\label{n=2}
The case $n=2$ is different because of the Euler identity.
We have the primary decomposition
\begin{equation}
\label{12/04/2013}
I_{2,2}=
\langle
a_{2}-b_{2}, 
a_{1}-b_{1}
\rangle
\cap 
\langle
a_{2}-b_{1},
a_{1}-b_{2}
\rangle
\cap 
\langle
a_{2}+b_{2}-1,
a_{1}+b_{1}-1
\rangle\cap
\langle 
a_{2}+b_{1}-1, a_{1}+b_{2}-1\rangle
\end{equation}
in the ring $\Q[a,b]$. Therefore, the equality (\ref{10d2012}) is valid  if and only if
\begin{equation}
\{b_1,b_2\}=\{a_1,a_2\} \hbox{ or } \{1-a_1,1-a_2\}. 
\end{equation}
The second possibility for $\{b_1,b_2\}$
is due to the Euler identity for the Gauss hypergeometric function   $_2F_1(a,b,c| z)=\sum_{k=0}^\infty \frac{(a)_k(b)_k}{k!(c)_k} z^k$: 
$$
_2F_1(a,b,c| z)=(1-z)^{c-a-b}\,  _2F_1(c-a,c-b,c| z)
$$
Note that we put $c=1$ and the same equality is valid for the logarithmic solution. This proves Theorem \ref{theo2} part 2 for $n=2$. 
\subsection{The symmetry}
\label{12apr2013}
Let $b_i=1-a_i,\ \ i=1,2,\ldots,n$ and let us restrict the ideal $I_{n,m}$ to this locus. 
For $n=3$ we have the primary decomposition
$$
I_{3,3}=
\langle
a_{2}+a_{3}-1,
2a_{1}-1
\rangle
\cap
\langle
2a_{3}-1,
a_{1}+a_{2}-1
\rangle\cap 
\langle
2a_{2}-1, 
a_{1}+a_{3}-1
\rangle
\cap
$$
$$
\langle
a_{3}-1, 
a_{2}-1, 
a_{1}-1
\rangle\cap
\langle
a_{3},
a_{2}, 
a_{1}
\rangle
$$
and for $n=4$ we have the primary decomposition
$$
I_{4,4}=
\langle
a_{2}+a_{3}-1, 
a_{1}+a_{4}-1
\rangle
\cap
\langle
a_{3}+a_{4}-1, 
a_{1}+a_{2}-1
\rangle
\cap
\langle
a_{2}+a_{4}-1, 
a_{1}+a_{3}-1
\rangle
\cap
$$
$$
\langle
a_{3}^{2}+a_{4}^{2}-2a_{3}-2a_{4}+2, 
a_{2}+a_{3}-2, 
a_{1}+a_{4}-2
\rangle
\cap
\langle
a_{3}^{2}+a_{4}^{2}, 
a_{2}+a_{3}, 
a_{1}+a_{4}
\rangle
\cap
$$
$$
\langle
a_{3}+a_{4}, 
a_{2}^{2}+a_{4}^{2}, 
a_{1}+a_{2}
\rangle
\cap
\langle
a_{3}^{2}+a_{4}^{2}-2a_{3}-2a_{4}+2, 
a_{2}+a_{4}-2, 
a_{1}+a_{3}-2
\rangle
\cap
$$
$$
\langle
a_{3}^{2}+a_{4}^{2}, 
a_{2}+a_{4}, 
a_{1}+a_{3}
\rangle
\cap
\langle
a_{3}+a_{4}-2, 
a_{2}^{2}+a_{4}^{2}-2a_{2}-2a_{4}+2, 
a_{1}+a_{2}-2
\rangle.
$$
We conclude that  Conjecture \ref{10d2012} is true for $n=3,4$ and $b_i=1-a_i$. 
Note that for $n=4$ the components which are not in the variety $\{a_i,\ i=1,2,3,4\}=\{1-a_i,\ i=1,2,3,4\}$ do not have $\Q$-rational points.

\subsection{Proof of Theorem \ref{maintheo}, part 2}
The only missing step is the verification of Conjecture \ref{10d2012}. The case $n=1$ can be done by hand and the case $n=2$
is done in \S\ref{n=2}. For $n\geq 3$ we do not get the primary decomposition of $I_{n,n}$ in Singular, therefore,  
we do not have a proof for Conjecture \ref{10d2012} with arbitrary parameters $a_i$ and $b_i$.
However, we may try to prove it for particular classes of parameters $a_i,b_i$.
The particular cases that appear in this section  are motivated  by 
the structure of $\dwork{p}$ described in Proposition 
\ref{27may2013}. 

First, let us take primes $p$  such that
$\dwork{p}(x)=1-x$, see Proposition \ref{27may2013} item 1.  In this case we have the primary decompositions in \S\ref{12apr2013} which implies that the set of $a_i$'s is invariant 
under $x\mapsto 1-x$. For $n=3$ we conclude that one of the parameters, let us say $a_3$,  is $\frac{1}{2}$ and $a_1=1-a_2$.
Now, restricted to $a_3=b_3=\frac{1}{2}$ we have the primary decomposition:
$$
I_{3,5}=
\langle
a_{2}-b_{1}, 
a_{1}-b_{2}
\rangle\cap
\langle
a_{2}-b_{2}, 
a_{1}-b_{1}
\rangle\cap
\langle
b_{2}-1, 
a_{2}+b_{1}-1, 
a_{1}-1
\rangle\cap
$$
$$
\langle
b_{2}-1, 
a_{2}-1, 
a_{1}+b_{1}-1
\rangle\cap
\langle
b_{1}-1, 
a_{2}+b_{2}-1, 
a_{1}-1
\rangle\cap
\langle
b_{1}-1, 
a_{2}-1, 
a_{1}+b_{2}-1
\rangle
$$
This proves Conjecture \ref{10d2012} in this particular case. Note that we use the fact that none of parameters is $1$.
  
For $n=4$, in a similar way , we can assume that $a_3=1-a_1$ and $a_4=1-a_2$. Again we do not  get the primary decomposition of $I_{4,4}$
restricted to the parameters $a_3=1-a_1, \ a_4=1-a_2, b_3=1-b_1, \ b_4=1-b_2$. We further use the structure of $\dwork{p}$ in order
to reduce the number of variables so that we can compute the primary decomposition of $I_{4,4}$.

For $q=2$ or $3$ fixed, let us consider the case in which  $q$ does not appear in the denominators of $a_i$'s. 
Using Proposition \ref{27may2013} item 2, we take primes $p$ such that $\dwork{p}(a_1)=qa_1-i$ and $\dwork{p}(a_2)=qa_2-j$, where $i,j$ are some integers
between $0$ and $q-1$.
 \begin{lemm}
\label{25may13}
 For $q=2,3$ and $i,j=0,1,\ldots,q-1$, the conjecture \ref{10d2012} is true for $n=4$ and 
$$
a_3=1-a_1, a_4=1-a_2,\ 
$$
$$
b_1=qa_1-i, b_2=qa_2-j,\ b_3=1-b_1,\ b_4=1-b_2 
$$
In each case the variety $V(I_{4,4})$ is a point with rational coordinates and if
we take  $0<a_i\leq \frac{1}{2}$ then  we get $7$ of $14$-hypergeometric cases.
\end{lemm}
\begin{proof}
The proof is purely computational. For $q=2$ or $3$ we have considered $q^2$ cases. In each case, we find that the variety 
$V(I_{4,4})$ consists only of one point. It has rational coordinates. If we assume that $0<a_i\leq \frac{1}{2}$ then for 
$q=2$ we have only the solutions $(a_1,a_2)=(\frac{1}{5},\frac{2}{5}),\ (\frac{1}{3},\frac{1}{3})$. For $q=3$ we get
$$
(a_1,a_2)= (\frac{1}{2}, \frac{1}{2}), (\frac{1}{4},\frac{1}{2}), ( \frac{1}{4}, \frac{1}{4}), (\frac{1}{5}, \frac{2}{5}), (\frac{1}{8}, \frac{3}{8}), 
(\frac{1}{10}, \frac{3}{10}).
$$
\end{proof}
Now, let us assume that $2$ and $3$ appears in the denominators of $a_i$'s. We are going to use Proposition \ref{27may2013} item 4 for $q=2,3$ and $m=1$.
For $q=2$, we take primes $p$ such that $\dwork{p}(a_i)=a_i+r_i,\ i=1,2$, where $\{r_1,r_2\}\subset\{0, \frac{1}{2},-\frac{1}{2}\}$. 
In a similar way, for $q=3$ we take primes $p$ such that $\delta_p(a_i)=a_i+s_i,\ i=1,2$, where 
$\{s_1,\ s_2\}\subset \{0, \frac{1}{3}, -\frac{2}{3}, \frac{2}{3}, -\frac{1}{3}  \}$. We can  assume that $r_1$ and $r_2$ (resp. $s_1$ and $s_2$) 
are not simultaneously zero. 
Let  $I_{4,4,2,r_1,r_2},$ be $I_{4,4}$ restricted to the parameters: 
$$
a_3=1-a_1, a_4=1-a_2,\ 
$$
$$
b_1=a_1+r_1, b_2=a_2+r_2,\ b_3=1-b_1,\ b_4=1-b_2,\ \  
$$
and in a similar way define $I_{4,4,3,s_1,s_2}$.  The variety $V(I_{4,4,2,r_1,r_2})\cap V(I_{4,4,3,s_1,s_2})$ is just one point. 
Moreover, it has rational coordinates.  
If we put the condition $0<a_i\leq \frac{1}{2}$ we get the remaining $7$ examples in the list (\ref{14cases}).

%
%
%
%
%

\subsection{Computational evidence  for conjectures \ref{main} and \ref{main2}}
\label{23may2013}
For $n=2$ and  $\{a_1,a_2\}=\{\frac{1}{2}(1\pm \frac{1}{m_1}-\frac{1}{m_2})\}$ with 
$m_1,m_2\in\N$ and $\frac{1}{m_1}+\frac{1}{m_2}>1$ the monodromy group of (\ref{14fev})
is a triangle group  of type $(m_1,m_2,\infty)$ and we have checked the Conjecture \ref{main}
for the truncated $q(a| z) \pmod{ q^{182 }}$ and for all the cases
$$
m_1\leq m_2\leq 24,\hbox{  or } m_1\leq 24,\ m_2=\infty. 
$$
and primes $2\leq p\leq 181$. For more computation of this type see the first author's homepage.
For Conjecture \ref{10d2012} (which implies Conjecture \ref{main2}), apart from its verification in particular cases
done in \S\ref{1june2013},  we have checked it for many other special loci in the parameter space $a,b$. The strategy is always to reduce 
the number of variables so that we can compute the primary decomposition of $I_{n,m}$ by a computer. The detailed discussion of this topic 
will be written somewhere  else.

\subsection{Proof of Corollary \ref{12june2013}}
\label{mtp}

%
%
%
Corollary \ref{12june2013} follows from Theorem \ref{maintheo}, part 2 for  $n=4$ in the case of good primes and for bad primes
follows from Theorem 1 in \cite{kr10}. Indeed the $k$th coefficient of $F$- 
 the holomorphic solution of any equation in table 1 for $n=4$- is a product of expressions like
$$
A_s(k):=\frac{(r_1/s)_k\cdots(r_m/s)_k}{k!^m}
$$
where $r_1,\cdots,r_m$ is a complete set of prime residues of $s$. Hence $m=\phi(s)=s(1-\frac{1}{p_1})\cdots(1-\frac{1}{p_l})$,
where $s=p_1^{\nu_1}\cdots p_l^{\nu_l}$, and one can easily check that  $N_s^kA_s(k)\in \Z$, where  $N_s=s^m\prod_{p\mid s}p^{m/p-1}$,
(see \cite{zu02} Lemma 1). The number $N=\prod_{s}N_s$ makes $F(Nz)$ with integer coefficients. Since the ring $\Z[[z]]$ is closed under the operator $\theta$, so
 $u_i(z)\in \Z[[z]]$ for $i=0,\cdots,4$.
For $u_5$ we proceed as follows. Since $u_1(z)\in 1+z\Z[[z]]$, 
it is enough to prove the statement for $W(a|z):=\frac{u_5(z)}{u_1(z)^2}=\theta (\frac{\losu}{F})$ (here we need to emphasize that $u_i$'s depend on $a$).
First let $p$ be a good prime. Acting $\theta$
on the both sides of  the congruence in Theorem \ref{9/11} we get 
$W(\dwork{p}(a)|z^p)-W(a|z)\in\Z_p[[z]]$. Since for $k=1,\cdots, p-1$, the $k$th coefficients of $F$ and $G$ are $p$-integral,
with a inductive argument we conclude that $W(a|z)\in\Z_p[[z]]$.
Now for bad primes we use Theorem 1 in \cite{kr10}, saying that $q(z)$ has integer coefficients.  Hence applying Lemma
\ref{dieudonne} for $q(z)$, implies Theorem \ref{9/11} for bad primes, with convention $\delta_p(a)=a$ and replacing $(z,z^p)$ by $(Nz,Nz^p)$ in the statement.
Then we can repeat the above argument to get integrality of $W(a|z)$. Now since $q(z)$ is integral, so converting all $u_i$'s in $q$-coordinate
remain them with integer coefficients. $\square$

The  most important modular object arising from the periods of Calabi-Yau varieties  is the Yukawa coupling: 
$$
Y:=n_0\frac{u_1^4}{(u_5+u_1^2)^3(1-z)}=n_0+\sum_{d=1}^\infty n_d d^3\frac{q^d}{1-q^d}
$$
Here, $n_0:=\int_M\omega^3$, where $M$ is the  $A$-model Calabi-Yau threefold of mirror symmetry
and $\omega$ is the K\"ahler form (the Picard-Fuchs equation (\ref{14fev}) is satisfied by the periods of B-model Calabi-Yau threefold).
The numbers $n_d$ are supposed to count the number of rational curves of degree $d$ in a generic $M$. For the 
case $a_i=\frac{i}{5},\ i=1,2,3,4$ few coefficients $n_d$ are given by $n_d=5, 2875, 609250, 317206375,\cdots$.

\def\cprime{$'$} \def\cprime{$'$} \def\cprime{$'$}


%

\end{document}